\documentclass{amsart}
\usepackage{amsmath}
\usepackage{amsfonts}
\usepackage{amssymb}
\usepackage{epsfig}
\newtheorem{theorem}{Theorem}[section]

\newtheorem{lm}[theorem]{Lemma}
\newtheorem{tr}[theorem]{Theorem}
\newtheorem{cor}[theorem]{Corollary}

\newtheorem{rem}[theorem]{Remark}
\newtheorem{pr}[theorem]{Proposition}
\newtheorem{example}[theorem]{Example}

\usepackage{color}

\begin{document}
	
\title{Orbits of actions of group superschemes}
\author[Bovdi]{V.A.~Bovdi}
\address{Department of Mathematical Sciences, UAEU, Al-Ain, United Arab Emirates}
\email{vbovdi@gmail.com}
\author[Zubkov]{A.N.~Zubkov}
\address{Department of Mathematical Sciences, UAEU, Al-Ain, United Arab Emirates; \linebreak Sobolev Institute of Mathematics, Omsk Branch, Pevtzova 13, 644043 Omsk, Russia}
\email{a.zubkov@yahoo.com}

\thanks{}

\begin{abstract}
Working over an algebraically closed field $\Bbbk$, we prove that all orbits of a left action of an algebraic group superscheme $G$ on a superscheme $X$ of finite type are locally closed. Moreover, such an orbit $Gx$, where $x$ is a $\Bbbk$-point of $X$, is closed if and only
if $G_{ev}x$ is closed in $X_{ev}$, or equivalently, if and only if $G_{res}x$ is closed in $X_{res}$. Here $G_{ev}$ is the largest purely even group super-subscheme of $G$ and $G_{res}$ is $G_{ev}$ regarded as a group scheme. Similarly, $X_{ev}$ is the largest purely even super-subscheme of $X$ and $X_{res}$ is $X_{ev}$ regarded as a scheme. We also prove that $\mathrm{sdim}(Gx)=\mathrm{sdim}(G)-\mathrm{sdim}(G_x)$, where $G_x$ is the stabilizer of $x$.
\end{abstract}

\maketitle
	
\section*{Introduction}

One of the fundamental properties of the action of a (not necessary affine) algebraic group scheme $G$ on a scheme $X$ of finite type is the existence  of closed orbits (see \cite[Proposition 7.6]{milne}, or \cite[II, \S 5, Proposition 3.2]{dg}).
In these notes we show that the similar property takes place for the action of a (not necessary affine) algebraic group superscheme $G$ on a superscheme $X$ of finite type.

Contrary to the purely even case we can not use directly the standard argument on orbits of minimal dimension. The difference is
that for a closed super-subscheme $Z$ of $X$, even if its underlying topological space $Z^e$ is $G(\Bbbk)$-stable, $Z$ is not necessary $G$-stable. It is clear why, $G(\Bbbk)$ is a group of $\Bbbk$-points of the largest purely even super-subgroup $G_{ev}$ of $G$. In general, $G$ is a product of $G_{ev}$ and some normal group subfunctor $N(G)$ (in terms of \cite{maszub3}, the latter is called the \emph{formal neighborhood} of the identity in $G$), and the action of $N(G)$ on $X$ is not controlled by $G(\Bbbk)$ at all.

To overcome this difficulty, we first prove that for any point $x\in X(\Bbbk)$ the superscheme morphism $\overline{a}_x:  G/G_x\to X$, induced
by the orbit morphism $a_x:  G\to X$, $g\mapsto gx$, is an immersion. In other words, the orbit $Gx$ is always a \emph{locally closed} super-subscheme of $X$. Therefore, $Gx$ is closed if and only if $(Gx)^e=(G_{ev}x)^e$ is.

The above two principal results are based on several auxiliary results that are interesting on their own. First, we introduce the  \emph{super-dimension} of certain Noetherian superschemes. This generalizes the previous definition from \cite{maszub}. We use the same notation $\mathrm{sdim}(X)$ as in \cite{maszub}. Then we prove that the super-dimension of a sheaf quotient $G/H$, where $G$ is an algebraic group superscheme and $H$ is its (closed) group super-subscheme, is equal to $\mathrm{sdim}(G)-\mathrm{sdim}(H)$.
The proof of this result is made possible by recent progress in the study of sheaf quotients of algebraic group superschemes (cf. \cite[Theorem 14.1]{maszub3}).

Recall that with any superscheme $X$ one can associate a \emph{graded} superscheme $\mathsf{gr}(X)$ as follows. The underlying topological space of $\mathsf{gr}(X)$ coincides with $X^e$ and the superalgebra sheaf of $\mathsf{gr}(X)$ is isomorphic to the sheafification of the presheaf
\[U\mapsto \oplus_{n\geq 0}\mathcal{I}_X(U)^n/\mathcal{I}_X(U)^{n+1}, \]
where $U$ runs over open subsets of $X^e$ and $\mathcal{I}_X$ is the superideal sheaf generated by $(\mathcal{O}_X)_1$. If $G$ is a (locally) algebraic group superscheme, then $\mathsf{gr}(G)$ has the natural structure of a (locally) algebraic group superscheme as well. Moreover, $G\to\mathsf{gr}(G)$ is an endofunctor of both categories of locally algebraic and algebraic group superschemes. We show that if $G$ is algebraic and $H$ is a closed group super-subscheme of $G$, then $\mathsf{gr}(G/H)\simeq \mathsf{gr}(G)/\mathsf{gr}(H)$. Besides, $(G/H)_{ev}\simeq G_{ev}/H_{ev}$.
These results  extend \cite[Proposition 4.18]{mastak} and \cite[Corollary 6.23]{sher}.

The article is organized as follows. In the first section we recall some elementary properties of super-commutative superalgebras. In the second section we recall the notion of \emph{Krull super-dimension} of a Noetherian superalgebra, whose even component is a Noetherian algebra of finite Krull dimension. The definition of Krull super-dimension reflects the notion of \emph{longest system of odd parameters}, which is more general and less restrictive than the long known notion of \emph{odd regular sequence} (see \cite{maszub, schmitt}).

The third section contains all the necessary facts about superschemes and group superschemes. In the fourth section we recall the definition of the category of Harish-Chandra pairs. It has been recently proven that this category is naturally equivalent to the category of locally algebraic group superschemes (see \cite[Theorem 12.10]{maszub3}, we also refer the reader to \cite{gav, masshib} for a better understanding of the development of this concept).

In the fifth section we recall the definition of left/right action of a group superscheme on a superscheme. For a given $\Bbbk$-point $x$ we also define the orbit morphism $a_x:  G\to X$ and the induced morphism $\overline{a}_x$.
The sixth section is devoted to the proving of the above mentioned isomorphisms $\mathsf{gr}(G/H)\simeq \mathsf{gr}(G)/\mathsf{gr}(H)$ and $(G/H)_{ev}\simeq G_{ev}/H_{ev}$. We also characterize locally algebraic group superschemes which isomorphic to a graded group superscheme. The content of seventh section is about the notion of super-dimension of certain Noetherian superschemes. In the eighth section we prove an analog of the well known formula for the dimension of a sheaf quotient $G/H$ in the category of algebraic group superschemes. In ninth section we prove the main Theorem \ref{orbit map is an immersion}. We conclude with some elementary example of an action of the odd unipotent group superscheme $G_a^-$ with respect to which all orbits are closed.

\section{Superalgebras}

Throughout this article $\Bbbk$ is a field of odd or zero characteristic. A $\mathbb{Z}_2$-graded $\Bbbk$-algebra $A$ is said to be a  \emph{superalgebra}. The homogeneous components of $A$ are denoted by $A_0$ and $A_1$. The elements of $A_0$ are called \emph{even} and the elements of $A_1$ are called \emph{odd}. We have the \emph{parity function} $(A_0\sqcup A_1)\setminus 0\to \mathbb{Z}_2, a\to |a|$, which maps $A_0\setminus 0$ to $0$ and $A_1\setminus 0$ to $1$, correspondingly.

A superalgebra $A$ is called \emph{super-commutative}, provided $ab=(-1)^{|a||b|}ba$ for any couple of homogeneous elements $a$ and $b$.  Throughout this article all superalgebras are supposed to be super-commutative, unless stated otherwise. The category of superalgebras (with graded morphisms) is denoted by $\mathsf{SAlg}_{\Bbbk}$.

Let $A$ be a superalgebra. Then a prime (maximal) superideal of $A$ has a form $\mathfrak{P}=\mathfrak{p}\oplus A_1$, where $\mathfrak{p}$ is a prime (respectively, maximal) ideal of $A_0$. In particular, $A$ is a \emph{local} superalgebra if and only if $A_0$ is a local algebra. A \emph{localization} $A_{\mathfrak{P}}$ is defined as $(A_0\setminus\mathfrak{p})^{-1}A$. It is clear that $A_{\mathfrak{P}}$ is a local superalgebra. If $A$ and $B$ are local superalgebras with maximal superideals $\mathfrak{P}$ and $\mathfrak{Q}$, then a superalgebra morphism
$\phi:  A\to B$ is said to be \emph{local} if $\phi(\mathfrak{P})\subseteq\mathfrak{Q}$.

Let $\Bbbk [x_1, \ldots, x_m\mid y_1, \ldots, y_n]$ denote a \emph{polynomial superalgebra}, freely generated by even and odd indeterminants
$x_1, \ldots, x_m$ and $y_1, \ldots, y_n$.

A superalgebra $A$ is called \emph{graded}, provided $A$ is $\mathbb{N}$-graded, say $A=\oplus_{n\geq 0} A(n)$, and
$A_0=\oplus_{n\geq 0} A(2n)$,  $A_1=\oplus_{n\geq 0} A(2n+1)$. We can associate with arbitrary superalgebra $A$ a graded superalgebra
$\mathsf{gr}(A)=\oplus_{n\geq 0} I^n_A/I^{n+1}_A$, where $I_A=AA_1$. The $0$-th component of $\mathsf{gr}(A)$ is denoted by
$\overline{A}$. We also say that a graded superalgebra $A$ is \emph{Grassman graded}, provided $A(n)=A(1)^n$ for all $n\geq 1$.
It is obvious that $A$ is Grassman graded if and only if $A\simeq\mathsf{gr}(A)$.
\begin{rem}\label{graded}
It is clear that $A\to\mathsf{gr}(A)$ is an endofunctor of the catgeory $\mathsf{SAlg}_{\Bbbk}$. Moreover, if $f:  A\to B$
is a morphism of Grassman graded superalgebras, i.e. $f(A(n))\subseteq B(n)$ for each $n\geq 0$, then we have a commutative diagram
\[\begin{array}{ccc}
\mathsf{gr}(A) & \simeq & A \\
\downarrow & & \downarrow \\
\mathsf{gr}(B) & \simeq & B
\end{array}  \]
whose vertical arrows are $\mathsf{gr}(f)$ and $f$ respectively.
\end{rem}

\section{Krull super-dimension}

Let $A$ be a \emph{Noetherian} superalgebra. In other words, $A_0$ is a Noetherian algebra and $A_1$ is a finitely generated $A_0$-module (cf. \cite[Lemma 1.4]{maszub}). We also assume that the Krull dimension $\mathrm{Kdim}(A_0)$ of $A_0$ is finite. Set $\mathrm{Kdim}(A_0)=n$.

A collection of odd elements $y_1, \ldots, y_k\in A_1$ is said to be a \emph{system of odd parameters} of $A$, provided there is
a longest prime chain $\mathfrak{p}_0\subseteq \cdots \subseteq\mathfrak{p}_n$ in $A_0$ such that $\mathrm{Ann}_{A_0}(y_1\cdots y_k)\subseteq
\mathfrak{p}_0$, or equivalently, $\mathrm{Kdim}(A_0/\mathrm{Ann}_{A_0}(y_1\cdots y_k))=\mathrm{Kdim}(A_0)$. The \emph{odd Krull dimension} of $A$ is defined as the length of a longest system of odd parameters of $A$ (see \cite[Section 4]{maszub}). It is denoted by $\mathrm{Ksdim}_1(A)$, as well as $\mathrm{Kdim}(A_0)$ is denoted by $\mathrm{Ksdim}_0(A)$ and called the \emph{even Krull dimension} of $A$. The couple of nonnegative integers $\mathrm{Ksdim}_0(A)\mid\mathrm{Ksdim}_1(A)$ is called just \emph{Krull super-dimension} of $A$ and denoted by $\mathrm{Ksdim}(A)$.

There is a more  restrictive notion of an \emph{odd regular sequence}. More precisely, the odd elements $y_1, \ldots, y_k\in A_1$ form an \emph{odd regular sequence}, provided 
$\mathrm{Ann}_A(y_1\cdots y_k)=Ay_1 +\cdots +Ay_k$ (see \cite[Corollary 3.1.2]{schmitt}). It is clear that any odd regular sequence is a system of odd parameters, but it is not necessary a part of some longest system of odd parameters (see \cite[Section 3]{zubkol}).

Let $A$ be local superalgebra with maximal superideal $\mathfrak{M}$. $A$ is said to be \emph{oddly regular} if one of the following equivalent conditions hold (\cite[Corollary 3.3]{schmitt}): 
\begin{enumerate}
\item[(i)] $I_A$ is generated by an odd regular sequence;
\item[(ii)] Every minimal base of $I_A$ is an regular sequence.
\end{enumerate}
Following \cite{schmitt} we denote $A_1/\mathfrak{m}A_1$ by $\Phi_A$.
\begin{lm}\label{oddly regular}
Let $\Bbbk (A)$ denote the residue field $A/\mathfrak{M}$. If $A$ is oddly regular, then $\mathrm{Ksdim}_1(A)=\dim_{\Bbbk (A)}(\Phi_A)$.	
\end{lm}
\begin{proof}
Note that $\dim_{\Bbbk (A)}(\Phi_A)$ equals the minimal number of generators of $A_0$-module $A_1$, and the minimal number of (odd) generators of $I_A$ as well. On the other hand, $\mathrm{Ksdim}_1(A)\leq \dim_{\Bbbk (A)}(\Phi_A)$ by \cite[Lemma 5.1]{maszub}. Thus if $A$ is oddly regular, then any minimal base of $I_A$ is the longest system of odd parameters.	
\end{proof}

\section{Superschemes}

Recall that a \emph{geometric superspace} $X$ consists
of a topological space $X^e$ and a sheaf of super-commutative superalgebras $\mathcal{O}_X$ on $X^e$, 
such that all stalks $\mathcal{O}_{X,x}$ for $x\in X^e$ are local superalgebras, whose maximal ideals are denoted by $\mathfrak{m}_x$. Define a \emph{residue field} at a point $x$ as $\kappa(x)=\mathcal{O}_{X, x}/\mathfrak{m}_x$.

A morphism of
superspaces $f:  X \to Y$ is a pair $(f^e, f^*)$, where $f^e:  X^e \to Y^e$ is a morphism
of topological spaces and $f^*:  \mathcal{O}_Y \to f^e
_* \mathcal{O}_X$ is a morphism of sheaves such that
$f^*_x:  \mathcal{O}_{Y,f^e(x)} \to \mathcal{O}_{X,x}$ is a local morphism for any $x\in X^e$. For any couple of open subsetes
$U\subseteq V\subseteq X^e$ let $\mathrm{res}_{V, U}$ denote the corresponding superalgebra morphism $\mathcal{O}_X(V)\to \mathcal{O}_X(U)$.

Let $R$ be a superalgebra. An \emph{affine geometric superscheme} $\mathrm{SSpec}(R)$ can be defined as
follows. The underlying topological space of $\mathrm{SSpec}(R)$ coincides with the prime
spectrum of $R_0$, endowed with the Zariski topology. For any open subset $U\subseteq (\mathrm{SSpec}(R))^e$, the super-ring $\mathcal{O}_{\mathrm{SSpec}(R)}(U )$ consists of all locally constant functions
$h:  U \to \sqcup_{\mathfrak{P}\in U} R_{\mathfrak{P}}$ such that $h(\mathfrak{P})\in R_{\mathfrak{P}}$ and  $\mathfrak{P}\in U$.

A superspace $X$ is called a \emph{(geometric) superscheme} if there is an open covering
$X^e = \cup_{i\in I}U_i$, such that each open super-subspace $(U_i, \mathcal{O}_X|_{U_i})$ is isomorphic to an affine
superscheme $\mathrm{SSpec}(R_i)$. Superschemes form a full subcategory of the category of geometric superspaces, denoted by
$\mathcal{SV}$.

A superscheme $X$ is said to be \emph{locally of finite type} over $\Bbbk$, if it can be covered by open affine super-subschemes $\mathrm{SSpec}(R_i)$ with each superalgebra $R_i$ to be finitely generated. If additionally this covering is finite, then $X$ is said to be
of \emph{finite type} (over $\Bbbk$).  Finally, a superscheme $X$ is said to be \emph{Noetherian}, if it can be covered by finitely many open affine super-subschemes $\mathrm{SSpec}(R_i)$, where each $R_i$ is a Noetherian superalgebra.

With any superscheme $X$ we can associate a \emph{largest purely even} closed super-subscheme $X_{ev}=(X^e, \mathcal{O}_X/\mathcal{I}_X)$ and a
\emph{graded} superscheme $\mathsf{gr}(X)=(X^e, \mathsf{gr}(\mathcal{O}_X))$, where
$\mathsf{gr}(\mathcal{O}_X)$ is a sheafification of the presheaf
\[
U\mapsto \oplus_{n\geq 0}\mathcal{I}_X(U)^n/\mathcal{I}_X(U)^{n+1},\qquad  U \subseteq X^e,
\]
and $\mathcal{I}_X$ is a superideal sheaf on $X$ generated by $(\mathcal{O}_X)_1$. In more detail, $\mathsf{gr}(\mathcal{O}_X)$ is a subsheaf of the (superalgebra) sheaf $\prod_{n\geq 0}\mathcal{I}_X^n/\mathcal{I}_X^{n+1}$ such that a section $s\in \prod_{n\geq 0}(\mathcal{I}_X^n/\mathcal{I}_X^{n+1})(U)$ belongs to $\mathsf{gr}(\mathcal{O}_X)(U)$ if and only if there is a covering of $U$ by open super-subschemes $U_i$ with $\mathrm{res}_{U, U_i}(s)\in \oplus_{n\geq 0}(\mathcal{I}_X^n/\mathcal{I}_X^{n+1})(U_i)$ for each index $i$ (cf. \cite[Section 10]{maszub3}). It is easy to see that $\mathsf{gr}(\mathrm{SSpec}(A))\simeq\mathrm{SSpec}(\mathsf{gr}(A))$.

The superscheme $X_{ev}$ can be regarded as an ordinary scheme, in which case it is denoted by $X_{res}$. Also, one can associate with $X$ a scheme $X_0=(X^e, (\mathcal{O}_X)_0)$.
So, we have two endofunctors $X\to \mathsf{gr}(X)$ and $X\to X_{ev}$ of the category $\mathcal{SV}$, and two functors $X\to X_{res}$ and $X\to X_0$ from $\mathcal{SV}$ to the category of schemes. For any superscheme morphism $f:  X\to Y$, the induced (super)scheme morphisms
$X_{ev}\to Y_{ev}$, $\mathsf{gr}(X)\to\mathsf{gr}(Y)$, $X_{res}\to Y_{res}$ and $X_0\to Y_0$ are denoted by
$f_{ev}$, $\mathsf{gr}(f)$, $f_{res}$ and $f_0$, respectively.

Recall that a superscheme morphism $f:  X\to Y$ is called an \emph{immersion} if there is an open super-subscheme $U$ of $Y$ such that
$f^e(X^e)\subseteq U^e$ and $f$ induces an isomorphism of $X$ onto a closed super-subscheme of $U$. This superscheme is said to be a \emph{locally closed} super-subscheme of $X$.

Equivalently, $f$ is an immersion if and only if $f^e$ is an homeomorphism of $X^e$ onto a closed subset of $U^e$ and the induced morphism of sheaves $f^*:  \mathcal{O}_Y|_U\to f_*\mathcal{O}_X$ is surjective (cf. \cite[Section 1]{maszub3}). The surjectivity of the sheaf morphism $f^*$ is equivalent to the condition that for any $x\in X^e$ the local superalgebra morphism $\mathcal{O}_{Y, f^e(x)}\to \mathcal{O}_{X, x}$ is surjective.

Let $\mathcal{F}$ denote the category of $\Bbbk$-functors from $\mathsf{SAlg}_{\Bbbk}$ to $\mathsf{Sets}$. There is a natural functor
from $\mathcal{SV}$ to $\mathcal{F}$ defined as
\[
X\to \mathbb{X}, \quad \mbox{where} \ \mathbb{X}(A)=\mathrm{Mor}_{\mathcal{SV}}(\mathrm{SSpec}(A), X),\qquad  A\in\mathsf{SAlg}_{\Bbbk}. 
\]
The $\Bbbk$-functor $\mathbb{X}$ is called a \emph{functor of points} of $X$.
For example, if $X\simeq \mathrm{SSpec}(R)$, then $\mathbb{X}\simeq \mathrm{SSp}(R)$, where
\[ 
\mathrm{SSp}(R)(A)=\mathrm{Mor}_{\mathsf{SAlg}_{\Bbbk}}(R, A),\qquad  A\in\mathsf{SAlg}_{\Bbbk}. 
\]
The $\Bbbk$-functor $\mathrm{SSp}(R)$ is called an \emph{affine superscheme}.

The functor $X\to\mathbb{X}$ is an equivalence between the category $\mathcal{SV}$ and a full subcategory of $\mathcal{F}$. The latter subcategory consists of all \emph{local} $\Bbbk$-functors covered by \emph{open subfunctors} isomorphic to affine superschemes (for more details see \cite[Theorem 10.3.7]{cf}, or \cite[Theorem 5.14]{maszub2}). It is is denoted by $\mathcal{SF}$ and the objects of $\mathcal{SF}$ are called \emph{just superschemes}.

In what follows we use the term \emph{superscheme} for both geometric superschemes and just superschemes, but we distinguish them by notations, i.e. geometric superschemes and their morphisms are denoted by $X, Y,\ldots,$ and $f, g, \ldots,$ and just superschemes and their morphisms (\emph{natural transformations} of $\Bbbk$-functors) are denoted by $\mathbb{X}, \mathbb{Y}, \ldots,$ and ${\bf f}, {\bf g}, \ldots,$ respectively. Similarly, a superscheme $\mathbb{X}$ is of (locally) finite type
over $\Bbbk$ or Noetherian if and only if its geometric counterpart $X$ is. Throughout this article what is proven for geometric superschemes can be naturally translated to the category of superschemes and vice versa.

Note that $f:  X\to Y$ is an open (closed) immersion if and only if ${\bf f}:  \mathbb{X}\to \mathbb{Y}$ is an open (closed) embedding (see \cite[Proposition 5.12]{maszub2} and \cite[Proposition 1.6]{maszub3}). Therefore, $f:  X\to Y$ is an immersion if and only ${\bf f}:  \mathbb{X}\to\mathbb{Y}$ is a natural isomorphism of $\mathbb{X}$ onto a closed subfunctor of an open subfunctor of $\mathbb{Y}$. As above, we call $\bf f$ an \emph{immersion} and $\mathbb{X}$ a \emph{locally closed} super-subscheme of $\mathbb{Y}$.

For example, the closed immersion $X_{ev}\to X$ corresponds
to the closed embedding $\mathbb{X}_{ev}\to \mathbb{X}$, where
\[ 
\mathbb{X}_{ev} (A) = \mathbb{X}(\iota)(\mathbb{X}(A_0))\simeq \mathbb{X}(A_0), \quad \mbox{for} \quad  A\in \mathsf{SAlg}_{\Bbbk}, 
\]
and $\iota:  A_0\to A$ is the natural (super)algebra monomorphism.

Let $\mathsf{gr}(\mathbb{X})$ denote the functor of points of $\mathsf{gr}(X)$. As above, $\mathsf{gr}(\mathrm{SSp}(A))\simeq\mathrm{SSp}(\mathsf{gr}(A))$.

Let $X$ be a superscheme and $U$ be an open super-subscheme of $X$. By \cite[Proposition 5.12]{maszub2} we have
$\mathcal{O}_X(U)\simeq\mathrm{Mor}_{\mathcal{SF}}(\mathbb{U}, \mathbb{A}^{1|1})$, where $\mathbb{A}^{1|1}\simeq\mathrm{SSp}(\Bbbk [x\mid y])$.
If it does not lead to confusion, we denote $\mathcal{O}_X(U)$ just by $\mathcal{O}(U)$, or by $\mathcal{O}(\mathbb{U})$ respectively.

A group object in the category of superschemes of locally finite type over $\Bbbk$ is called a \emph{locally algebraic group superscheme}.
Respectively, a group object in the category of superschemes of finite type over $\Bbbk$ is called an \emph{algebraic group superscheme}.
If $G$ (respectively, $\mathbb{G}$) is a (locally) algebraic group superscheme, then
$\mathsf{gr}(G)$ (respectively, $\mathsf{gr}(\mathbb{G})$) is a (locally) algebraic group superscheme as well.

Recall that with any group superscheme $G$ one can associate a \emph{largest affine factor-group superscheme}
$G^{aff}$ such that any group superscheme morphism $G\to H$, with $H$ to be affine, factors through $G\to G^{aff}$ (cf. \cite[Corollary 5.3]{maszub3}). By the equivalence of categories, the functor of points of $G^{aff}$ is the largest  affine factor-group superscheme of $\mathbb{G}$, and it is denoted by $\mathbb{G}^{aff}$. If $G$ (respectively, $\mathbb{G}$) is algebraic, then $G^{aff}\simeq\mathrm{SSpec}(\mathcal{O}(G))$ (respectively, $\mathbb{G}^{aff}\simeq\mathrm{SSp}(\mathcal{O}(\mathbb{G}))$).

For any $g\in \mathbb{G}(\Bbbk)$ we have a natural transformation $l_g:  \mathbb{G}\to\mathbb{G}$, defined as $x\mapsto gx$, where  $x\in\mathbb{G}(A)$ and $A\in\mathsf{SAlg}_{\Bbbk}$. We call it a \emph{left shift transformation} (by $g$). The \emph{right shift transformation} $r_g$ is defined symmetrically. Since both left and right shifts are  natural isomorphisms of $\mathbb{G}$ onto itself, they take closed (open) subfunctors to closed (open) subfunctors as well. In particular, for any closed group super-subscheme $\mathbb{H}$ of $\mathbb{G}$ the \emph{left and right cosets} $l_g(\mathbb{H})$ and
$r_g(\mathbb{H})$ are closed super-subschemes in $\mathbb{G}$. We denote them by $g\mathbb{H}$ and $\mathbb{H}g$ respectively.

\section{Harish-Chandra pairs and the fundamental equivalence}

For the content of this section we refer to \cite{maszub3}.

Recall that a couple $(\mathsf{G}, \mathsf{V})$ is called a \emph{Harish-Chandra pair}, provided $\mathsf{G}$ is a purely even locally algebraic group superscheme (which can be regarded as group scheme as well), $\mathsf{V}$ is a finite dimensional $\mathsf{G}$-(super)module and the following conditions hold:
\begin{enumerate}
\item[(a)] there is a symmetric bilinear map $\mathsf{V}\times\mathsf{V}\to\mathrm{Lie}(\mathsf{G}), \ (v, w)\mapsto [v, w]$, $v, w\in\mathsf{V}$,  that is $\mathsf{G}$-equivariant with respect to the diagonal action of $\mathsf{G}$ on $\mathsf{V}\times\mathsf{V}$ and the adjoint action of $\mathsf{G}$ on $\mathrm{Lie}(\mathsf{G})$;
\item[(b)] the induced action of $\mathrm{Lie}(\mathsf{G})$ on $\mathsf{V}$ satisfies $[v, v]\cdot v=0$ for any $v\in\mathsf{V}$.
\end{enumerate}
Harish-Chandra pairs form a category with morphisms $(\mathsf{f}, \mathsf{u}):  (\mathsf{G}, \mathsf{V})\to (\mathsf{H}, \mathsf{W})$, where
$\mathsf{f}:  \mathsf{G}\to \mathsf{H}$ is a morphism of group schemes and $\mathsf{u}:  \mathsf{V}\to \mathsf{W}$ is a morphism of $\mathsf{G}$-modules ($\mathsf{W}$ is regarded as a $\mathsf{G}$-module via $\mathsf{f}$). Besides, the diagram
\[
\begin{array}{ccc}
\mathsf{V}\times\mathsf{V} & \to & \mathrm{Lie}(\mathsf{G}) \\
\downarrow & & \downarrow \\
\mathsf{W}\times\mathsf{W} & \to & \mathrm{Lie}(\mathsf{H})
\end{array} 
\]
is commutative, where the first vertical map is $\mathsf{u}\times\mathsf{u}$ and the second vertical map
is the differential of $\mathsf{f}$. The category of Harish-Chandra pairs is denoted by $\mathsf{HCP}$.

There is a functor from the category of locally algebraic group superschemes to the category of Harish-Chandra pairs
\[
G\to (G_{ev}, \mathrm{Lie}(G)_1)  
\]
or
\[
\mathbb{G}\to (\mathbb{G}_{ev}, \mathrm{Lie}(\mathbb{G})_1).
\]
The action of $G_{ev}$ (respectively, of $\mathbb{G}_{ev}$) on
$\mathrm{Lie}(G)_1$ (respectively, on $\mathrm{Lie}(\mathbb{G})_1$) is the restriction of the adjoint action of
$G$ (respectively, of $\mathbb{G}$) on $\mathrm{Lie}(G)$ (respectively, on $\mathrm{Lie}(\mathbb{G})$).
This functor is called \emph{Harish-Chandra functor} and denoted by $\Phi$.
\begin{tr}(see \cite[Theorem 12.10]{maszub3})
The Harish-Chandra functor is an equivalence.	
\end{tr}	
The quasi-inverse of $\Phi$ is denoted by $\Psi$. For any Harish-Chandra pair $(\mathsf{G}, \mathsf{V})$ the group superscheme $\mathbb{G}=\Psi((\mathsf{G}, \mathsf{V}))$ is constructed as follows.

Let $\Bbbk [\varepsilon]$ denote the algebra of \emph{dual numbers}, i.e. this is the
algebra generated by the element $\varepsilon$ subject to the relation $\varepsilon^2=0$. Then we have an exact sequence of groups
\[
1\to \mathrm{Lie}(\mathsf{G}) \to \mathsf{G}(\Bbbk [\varepsilon])\stackrel{\mathsf{G}(p)}{\to} \mathsf{G}(\Bbbk)\to 1,    
\]
where $p:  \Bbbk [\varepsilon]\to \Bbbk$ is the algebra morphism such that $p(\varepsilon)=0$ (cf. \cite[II, \S 4]{dg}). The image of $x\in \mathrm{Lie}(\mathsf{G})$ in $\mathsf{G}(\Bbbk [\epsilon])$ is denoted by $e^{\varepsilon x}$. For any superalgebra $A$ and $b\in A_0$ such that $b^2=0$ there is the unique algebra morphism $\alpha_b:  \Bbbk [\varepsilon]\to A_0$ that takes $\varepsilon$ to $b$. The element
$\mathsf{G}(\alpha_b)(e^{\varepsilon x})\in\mathsf{G}(A_0)$ is denoted by $f(b, x)$. Let $\mathrm{F}(A)$ denote a free group, freely generated by the formal elements $e(a, v)$,  where  $a\in A_1$, $v\in\mathsf{V}$.

Then the group $\mathbb{G}(A)$ is isomorphic to the factor-group of the free product $\mathsf{G}(A_0)*\mathrm{F}(A)$ modulo the relations
\begin{enumerate}
\item[(i)] $[e(a, v), e(a', v')] = f (-aa', [v, v'])$;
\item[(ii)] $[e(a, v), f(b, x)]=e(ab, [v, x])$;
\item[(iii)] $[f(b, x), f(b', x')]=f(bb', [x, x'])$;
\item[(iv)] $e(a, v)e(a', v)=f(-aa', \frac{[v, v]}{2})e(a+a', v)$;
\item[(v)] $e(a, v)^g=e(a, g\cdot v), \quad  f(b, x)^g=f(b, g\cdot x)$ \qquad for any   $g\in\mathsf{G}(A_0)$;
\item[(vi)]$e(a, cv)=e(ca, v)$ \qquad for any scalar $c\in\Bbbk$. 
\end{enumerate}
Here the group commutator $[u, v]$ is defined as $uv u^{-1}v^{-1}$ and $u^v$ denotes $vuv^{-1}$. A superalgebra morphism $\beta:  A\to A'$ induces the natural group morphism $\mathsf{G}(A_0)*\mathrm{F}(A)\to \mathsf{G}(A'_0)*\mathrm{F}(A')$ that takes $e(a, v)$ to $e(\beta(a), v)$, hence it preserves the relations (i)-(v). Thus $\mathbb{G}$ is obviously a group functor. Moreover, by \cite[Corollary 12.4 and Theorem 12.5]{maszub3} any element
of $\mathbb{G}(A)$ can be uniquely expressed as the product
\[
g e(a_1, v_1)\cdots e(a_t, v_t), 
\]
where $g\in\mathsf{G}(A_0)$ and $v_1, \ldots, v_t$ is a fixed basis of $\mathsf{V}$. In particular, $\mathbb{G}$ is isomorphic to
$\mathsf{G}\times \mathrm{SSp}(\Lambda(\mathsf{V}^*))$ as a superscheme.

\section{Actions, orbits and quotients}

Let $\mathbb{G}$ be a group superscheme and $\mathbb{X}$ be a superscheme. The functor morphism ${\bf f}:  \mathbb{X}\times\mathbb{G}\to \mathbb{X}$
defines a \emph{right action} of $\mathbb{G}$ on $\mathbb{X}$, provided the following properties hold:
\begin{enumerate}
\item ${\bf f}(A)(x, 1)=x$;
\item ${\bf f}(A)({\bf f}(A)(x, g), h)={\bf f}(A)(x, gh),\qquad  x\in\mathbb{X}(A), \ g,h\in\mathbb{G}(A),\ A\in\mathsf{SAlg}_{\Bbbk}$.	
\end{enumerate}
A left action is defined similarly.

In what folllows we denote ${\bf f}(A)(x, g)$ just by $xg$ (respectively, $gx$ for a left action), if it does not lead to confusion. A given action is called \emph{free}, whenever $xg=x$ (respectivelt, $gx=x$) implies $g=e$ for any $x\in\mathbb{X}(A), g\in\mathbb{G}(A)$.

With each (right) free action of $\mathbb{G}$ on $\mathbb{X}$ one can associate the functor
\[ 
A\mapsto \mathbb{X}(A)/\mathbb{G}(A),\qquad  A\in\mathsf{SAlg}_{\Bbbk}, 
\]
which is called a \emph{naive quotient} and denoted by $(\mathbb{X}/\mathbb{G})_{(n)}$. The sheafification of this functor in the \emph{Grothendieck topology of fppf coverings} is denoted by $\mathbb{X}/\mathbb{G}$ (see \cite{maszub2} for more details). The functor $(\mathbb{X}/\mathbb{G})_{(n)}$ is a \emph{dense subfunctor}  of
$\mathbb{X}/\mathbb{G}$ in the following sense. For any superalgebra $A$ and any $x\in (\mathbb{X}/\mathbb{G})(A)$ there is a \emph{finitely presented} $A$-superalgebra $A'$, which is a faithfully flat $A$-supermodule, such that $(\mathbb{X}/\mathbb{G})(\iota_A)(x)$ belongs to
$\mathbb{X}(A')/\mathbb{G}(A')$. Here $\iota_A:  A\to A'$ is the corresponding (injective) superalgebra morphism. We call such superalgebra $A'$ a \emph{fppf covering} of $A$ as well.

In general, $\mathbb{X}/\mathbb{G}$ is no longer superscheme. But if $\mathbb{X}$ is an algebraic group superscheme and $\mathbb{G}$ is its closed group super-subscheme, naturally acting on $\mathbb{X}$ on the right, then $\mathbb{X}/\mathbb{G}$ is a superscheme of finite type over $\Bbbk$. Moreover, the quotient morphism
$\mathbb{X}\to \mathbb{X}/\mathbb{G}$ is faithfully flat (see \cite[Theorem 14.1]{maszub3}).

Another case when $\mathbb{X}/\mathbb{G}$ is a superscheme is the following. Let $\mathbb{X}$ be an affine superscheme and $\mathbb{G}$ is a finite (hence affine) group superscheme, acting freely on $\mathbb{X}$. Then $\mathbb{X}/\mathbb{G}$ is an affine superscheme (see \cite{zub1}). This statement is partially covered by the recent remarkable result \cite[Theorem 1.8]{masoeyuta}. The authors proved that if $\mathbb{X}$ is a locally Noetherian superscheme and $\mathbb{G}$ is
affine, then $\mathbb{X}/\mathbb{G}$ is a superscheme if and only if $\mathbb{X}_{ev}/\mathbb{G}_{ev}$ is a scheme.

Assume that $\mathbb{G}$ acts on $\mathbb{X}$ on the left. Take $x\in\mathbb{X}(\Bbbk)$ and consider the \emph{orbit morphism} ${\bf a}_x:  \mathbb{G}\to\mathbb{X}$,\quad  $g\mapsto gx$, where $g\in\mathbb{G}(A)$ and  $A\in\mathsf{SAlg}_{\Bbbk}$. The stabilizer $\mathbb{G}_x$ is a group subfunctor of $\mathbb{G}$ defined as
\[ 
\mathbb{G}_x(A)=\{g\in \mathbb{G}(A)\mid gx=x  \}, \qquad  A\in\mathsf{SAlg}_{\Bbbk}. 
\]
The subfunctor $\mathbb{G}_x$ is a closed group super-subscheme of $\mathbb{G}$. In fact, let $\mathbb{Y}$ denote the subfunctor of $\mathbb{X}$ such that $\mathbb{Y}(A)=\{x\}$ for any superalgebra $A$. By \cite[Lemma 1.1]{maszub3} $\mathbb{Y}$ is closed in $\mathbb{X}$, hence $\mathbb{G}_x={\bf a}_x^{-1}(\mathbb{Y})$ is closed in $\mathbb{G}$.

Further, $(\mathbb{G}/\mathbb{G}_x)_{(n)}$ is a subfunctor of $\mathbb{X}$. Its sheafifcation $\mathbb{G}/\mathbb{G}_x$ can be naturally identified with a \emph{sub-faisceau} of the \emph{faisceau} $\mathbb{X}$ (see \cite[Section 2]{zub2} for more details). In other words, ${\bf a}_x$ factors through the embeding $\mathbb{G}/\mathbb{G}_x\to \mathbb{X}$. The image of $\mathbb{G}/\mathbb{G}_x$ in $ \mathbb{X}$ is regarded as the $\mathbb{G}$-\emph{orbit} of $x$. We also denote it by $\mathbb{G}x$. Note that $\mathbb{G}(A)x$ is a proper subset of $(\mathbb{G}x)(A)$ in general. If $\mathbb{G}$ is an algebraic group superscheme, then $\mathbb{G}x\simeq \mathbb{G}/\mathbb{G}_x$ is a superscheme of finite type.

In terms of geometric superschemes, $x\in\mathbb{X}(\Bbbk)$ corresponds to a superscheme morphism $\mathrm{SSpec}(\Bbbk)\to X$. The latter morphism is uniquely defined by a point in $X^e$, which can be denoted by the same symbol $x$, such that $\kappa(x)=\Bbbk$. By the equivalence of categories, $G$ acts on $X$ on the left and we have the orbit morphism $a_x:  G\to X$ that
coincides with the composition $G\simeq G\times\mathrm{SSpec}(\Bbbk)\to G\times X\to X$ of superscheme morphisms.
Finally, if $G$ is an algebraic group superscheme, then $a_x$ factors through the corresponding morphism $G/G_x\to X$ of superschemes of finite type.

\section{Graded group superschemes}

A locally algebraic group superscheme $G$ is called \emph{graded} if there is a locally algebraic group superscheme $H$ such that $G\simeq\mathsf{gr}(H)$. The following proposition refines \cite[Proposition 11.1]{maszub3} and extends \cite[Lemma-Definition 3.9]{mastak}.

\begin{pr}\label{refine}
Let $G$ be a locally algebraic group superscheme. Set $\Phi(G)=(\mathsf{G}, \mathsf{V})$ or, equivalently, $\Psi((\mathsf{G}, \mathsf{V}))\simeq G$. The following statements are equivalent:
\begin{enumerate}
\item[(i)] $G$ is graded;
\item[(ii)] $G_{ev}\to G$ is split in the category of group superschemes;
\item[(iii)] The bilinear map $\mathsf{V}\times \mathsf{V}\to \mathrm{Lie}(\mathsf{G})$ is zero;
\item[(iv)] The Lie super-bracket on $\mathrm{Lie}(G)$, restricted to $\mathrm{Lie}(G)_1\times \mathrm{Lie}(G)_1$, vanishes.
\end{enumerate}	
\end{pr}
\begin{proof}
The equivalence (iii)$\leftrightarrow$(iv) is obvious.
	
If $G\simeq\mathsf{gr}(H)$, then $G_{ev}\simeq \mathsf{gr}(H)_{ev}\simeq H_{ev}$. Moreover, there is a group superscheme morphism $q:  \mathsf{gr}(H)\to H_{ev}$ whose composition with the embedding $H_{ev}\to \mathsf{gr}(H)$ is the identity morphism (see \cite[Section 11]{maszub3}). This implies (i)$\rightarrow$(ii).	

In the language of Harish-Chandra pairs $G_{ev}\to G$ is split if and only if there is a morphism of pairs
$(\mathsf{f}, 0):  (\mathsf{G}, \mathsf{V})\to (\mathsf{G}, 0)$ such that its composition with $(\mathrm{id}_{\mathsf{G}}, 0):  (\mathsf{G}, 0)\to (\mathsf{G}, \mathsf{V})$ is again $(\mathrm{id}_{\mathsf{G}}, 0)$. Thus $\mathsf{f}=\mathrm{id}_{\mathsf{G}}$ and the commutative diagram
\[
\begin{array}{ccc}
\mathsf{V}\times\mathsf{V} & \to & \mathrm{Lie}(\mathsf{G}) \\
\downarrow & & \parallel \\
0\times 0 & \to & \mathrm{Lie}(\mathsf{G})
\end{array} 
\]
infers (ii)$\leftrightarrow$(iii).

If (iii) holds, then $\mathbb{G}$ is a semi-direct product of $\mathbb{G}_{ev}$ and a normal abelian purely odd unipotent group super-subscheme $\mathbb{G}_{odd}$ (cf. \cite[Proposition 11.1]{maszub3}). More precisely, for any $A\in\mathsf{SAlg}_{\Bbbk}$ the group $\mathbb{G}_{odd}(A)$ is generated by
the elements $e(a, v)$, which are pairwise commuting and satisfy $e(a, v)e(a', v)=e(a+a', v)$. In other words, $\mathbb{G}_{odd}\simeq\mathrm{SSp}(\Lambda(\mathsf{V}^*))$, where all elements of $\mathsf{V}^*$ are supposed to be primitive in
the Hopf superalgebra $\Lambda(\mathsf{V}^*)$.

 It remains to show that $\mathbb{G}\simeq\mathsf{gr}(\mathbb{G})$, or, equivalently, $G\simeq\mathsf{gr}(G)$.
The superscheme $\mathbb{G}\simeq \mathsf{G}\times\mathrm{SSp}(\Lambda(\mathsf{V}^*))$ can be covered by open affine super-subschemes of the form $\mathbb{U}=\mathsf{U}\times\mathrm{SSp}(\Lambda(\mathsf{V}^*))$, where $\mathsf{U}$ runs over open affine subschemes of $\mathsf{G}$.

The multiplication morphism ${\bf m}:  \mathbb{G}\times\mathbb{G}\to\mathbb{G}$ is defined as
\[
(g, e)\times (g', e')\mapsto (g g')\times (e^{g'^{-1}} e'),\qquad  g, g'\in\mathbb{G}_{ev}, e, e'\in \mathbb{G}_{odd}. 
\]
The conjugation action of $\mathbb{G}_{ev}$ on $\mathbb{G}_{odd}\simeq \mathrm{SSp}(\Lambda(\mathsf{V}^*))$ is uniquely defined by the induced action of $\mathbb{G}_{ev}\simeq \mathsf{G}$ on $\mathsf{V}^*$, hence it factors through $\mathbb{G}_{ev}\to\mathbb{G}_{ev}^{aff}\simeq \mathrm{SSp}(\mathcal{O}(\mathsf{G}))$ (cf. \cite[Corollary 5.3 and Lemma 8.1]{maszub3}). Let $\tau$ be the corresponding $\mathcal{O}(\mathsf{G})$-comodule map $\mathsf{V}^*\to\mathsf{V}^*\otimes \mathcal{O}(\mathsf{G}), \tau(v^*)=v^*_{(1)}\otimes v^*_{(2)}, v^*\in\mathsf{V}^*$, where we omit the symbol of summation in Sweedler's notation for $\tau$.

Let $\mathsf{m}$ denote the multiplication morphism $\mathsf{G}\times\mathsf{G}\to\mathsf{G}$. For any triple $\mathsf{U}, \mathsf{U}', \mathsf{U}''$ of open affine subschemes of $\mathsf{G}$, the open subscheme $\mathsf{m}^{-1}(\mathsf{U}'')\cap (\mathsf{U}\times \mathsf{U}')$
of $\mathsf{G}\times\mathsf{G}$ can be covered by open subschemes of the form $(\mathsf{U}\times\mathsf{U}')_f\simeq \mathrm{Sp}((\mathcal{O}(\mathsf{U})\otimes\mathcal{O}(\mathsf{U}'))_f), f\in \mathcal{O}(\mathsf{U})\otimes\mathcal{O}(\mathsf{U}')$.
Since $\mathsf{m}$ can be glued from all $\mathsf{m}|_{(\mathsf{U}\times\mathsf{U}')_f}:  (\mathsf{U}\times\mathsf{U}')_f\to \mathsf{U}''$, the multiplication map ${\bf m}:  \mathbb{G}\times \mathbb{G}\to\mathbb{G}$ can be glued from  all $(\mathbb{U}\times\mathbb{U}')_f\to \mathbb{U}''$.
More precisely, each morphism $(\mathbb{U}\times\mathbb{U}')_f\to \mathbb{U}''$ is dual to
the superalgebra morphism $\mathcal{O}((\mathbb{U}\times\mathbb{U}')_f)\to\mathcal{O}(\mathbb{U}'')$, defined by
\[
\textstyle
x\otimes y\mapsto \frac{x_{(1)}\otimes x_{(2)}(y_{(1)})_{(2)}}{f^s}\otimes (y_{(1)})_{(1)}\otimes y_{(2)}, 
\]
where $x\in\mathcal{O}(\mathsf{U}''), y\in\Lambda(\mathsf{V}^*)$, the algebra morphism
\[
\textstyle
x\mapsto  \frac{x_{(1)}\otimes x_{(2)}}{f^s} \qquad\quad  (s \ \mbox{depends on } \ x)  \]
is dual to $\mathsf{m}|_{(\mathsf{U}\times\mathsf{U}')_f}$, and
\[
y\mapsto y_{(1)}\otimes y_{(2)}, \qquad  y_{(1)}\mapsto (y_{(1)})_{(1)}\otimes (y_{(1)})_{(2)}  
\]
are the comultiplication in $\Lambda(\mathsf{V}^*)$ and $(\mathrm{id}_{\Lambda(\mathsf{V}^*)}\otimes \mathrm{res}_{\mathsf{G}, \mathsf{U}'})\wedge(\tau)$ respectively.  Thus obviously follows that $\mathcal{O}((\mathbb{U}\times\mathbb{U}')_f)\to\mathcal{O}(\mathbb{U}'')$ is a morphism of graded superalgebras. Remark \ref{graded} combined with \cite[Proposition 10.3(1)]{maszub3} conclude the proof.
\end{proof}
The following proposition extends \cite[Proposition 4.18]{mastak} and \cite[Corollary 6.23]{sher}.
\begin{pr}\label{graded quotient}
Let $G$ be an algebraic group superscheme and $H$ be its closed group super-subscheme. Then
$\mathsf{gr}(G/H)\simeq \mathsf{gr}(G)/\mathsf{gr}(H)$ and $(G/H)_{ev}\simeq G_{ev}/H_{ev}$.
\end{pr}
\begin{proof}
We use again the language of $\Bbbk$-functors and follow the steps in the proof of \cite[Theorem 14.1]{maszub3}. Set also $\Phi(\mathbb{H})=(\mathsf{H}, \mathsf{W})$.

Observe that if $\Phi(\mathbb{G})=(\mathsf{G}, \mathsf{V})$, then $\Phi(\mathsf{gr}(\mathbb{G}))
=(\mathsf{G}, \mathsf{V})$ as well, with the only difference that the bilinear map $\mathsf{V}\times\mathsf{V}\to\mathrm{Lie}(\mathsf{G})$ is the zero map. Besides, for any group superscheme morphism ${\bf f}:  \mathbb{G}\to \mathbb{H}$ there is $\Phi({\bf f})=\Phi(\mathsf{gr}({\bf f}))$. By \cite[Theorem 12.11]{maszub3} our proposition follows for any normal $\mathbb{H}$. If $\mathbb{L}$ is a normal closed group super-subscheme of $\mathbb{G}$, such that $\mathbb{L}\leq\mathbb{H}$ and our proposition holds for $\mathbb{G}/\mathbb{L}$ and $\mathbb{H}/\mathbb{L}$, then
\[
\begin{split}
\mathsf{gr}(\mathbb{G}/\mathbb{H})&\simeq \mathsf{gr}((\mathbb{G}/\mathbb{L})/(\mathbb{H}/\mathbb{L}))\simeq \mathsf{gr}(\mathbb{G}/\mathbb{L})/\mathsf{gr}(\mathbb{H}/\mathbb{L})\\
&\simeq
(\mathsf{gr}(\mathbb{G})/\mathsf{gr}(\mathbb{L}))/(\mathsf{gr}(\mathbb{H})/\mathsf{gr}(\mathbb{L}))\simeq \mathsf{gr}(\mathbb{G})/\mathsf{gr}(\mathbb{H}). 
\end{split}
\]
Arguing as in \cite[Lemma 14.2]{maszub3}, one can reduce the general case to the case when $\mathbb{G}$ and $\mathbb{H}$ satisfy the following conditions:
\begin{enumerate}
\item[(a)] $\mathsf{G}=\mathsf{M}\times\mathsf{A}(\mathsf{G})$, where $\mathsf{A}(\mathsf{G})=\ker(\mathsf{G}\to \mathsf{G}^{aff})$ is a central anti-affine group subscheme of $\mathsf{G}$ (cf. \cite[Corollary 8.14 and Proposition 8.37]{milne});
\item[(b)] $\mathsf{M}$ is an affine group subscheme of $\mathsf{G}$ such that $\mathsf{H}\leq\mathsf{M}$.
\end{enumerate}
On can choose a covering of $\mathsf{G}$ by open affine $\mathsf{H}$-saturated subschemes $\mathsf{U}=\mathsf{U}_{aff}\times \mathsf{U}_{ab}$, such that
$\mathsf{U}_{aff}$ form an open covering of $\mathsf{M}$ and $\mathsf{U}_{aff}/\mathsf{H}$ form an covering of $\mathsf{M}/\mathsf{H}$ by open affine subschemes. Thus $\mathsf{U}/\mathsf{H}$ form an open covering of $\mathsf{G}/\mathsf{H}$ by open affine subschemes.
Next, each $\mathbb{V}=\mathsf{U}\times\mathrm{SSpec}(\mathsf{V}^*) $ is an open affine $\mathbb{H}$-stable (or $\mathbb{H}$-saturated) super-subscheme of $\mathbb{G}$ (see \cite{maszub3}, Lemma 14.5). The superscheme $\mathbb{G}/\mathbb{H}$ is covered by open affine super-subschemes $\mathbb{V}/\mathbb{H}$ and similarly, $\mathsf{gr}(\mathbb{G})/\mathsf{gr}(\mathbb{H})$ is covered by open affine super-subschemes $\mathsf{gr}(\mathbb{V})/\mathsf{gr}(\mathbb{H})$.

As a right $\mathcal{O}(\mathbb{H})$-coideal superalgebra $\mathcal{O}(\mathbb{V})$ is isomorphic to the cotensor product $(\mathcal{O}(\mathsf{U})\otimes\Lambda(\mathsf{Z}))\square_{\mathcal{O}(\mathsf{H})}\mathcal{O}(\mathbb{H})$ of the right and left $\mathcal{O}(\mathsf{H})$-coideal superalgebras $\mathcal{O}(\mathsf{U})\otimes\Lambda(\mathsf{Z})$ and $\mathcal{O}(\mathbb{H})$ respectively, where $\mathsf{Z}=\ker(\mathsf{V}^*\to \mathsf{W}^*)$. In particular, the superalgebra $\mathcal{O}(\mathbb{V}/\mathbb{H})\simeq \mathcal{O}(\mathbb{V})^{\mathcal{O}(\mathbb{H})}$ is isomorphic to $(\mathcal{O}(\mathsf{U})\otimes\Lambda(\mathsf{Z}))^{\mathcal{O}(\mathsf{H})}$. For another open affine subscheme
$\mathsf{U}'=\mathsf{U}'_{aff}\times\mathsf{U}'_{ab}$ and for the corresponding open affine $\mathbb{H}$-saturated super-subscheme
$\mathbb{V}'=\mathsf{U}'\times\mathrm{SSpec}(\mathsf{V}^*) $ we have
\[(\mathbb{V}\cap\mathbb{V}')/\mathbb{H}= \mathbb{V}/\mathbb{H}\cap \mathbb{V}'/\mathbb{H}\]
This open super-subcheme is glued from the open affine super-subschemes $\mathbb{T}/\mathbb{H}$, where
$\mathbb{T}=\mathsf{T}\times\mathrm{SSpec}(\mathsf{V}^*)$ and the $\mathsf{H}$-saturated affine subschemes $\mathsf{T}$ form an open covering of $\mathsf{U}\cap\mathsf{U}'$. In other words, the superscheme structure of $\mathbb{G}/\mathbb{H}$ is uniquely defined by the family
of superalgebras $(\mathcal{O}(\mathsf{U})\otimes\Lambda(\mathsf{Z}))^{\mathcal{O}(\mathsf{H})}$ and by the family of superalgebra morphisms $(\mathcal{O}(\mathsf{U})\otimes\Lambda(\mathsf{Z}))^{\mathcal{O}(\mathsf{H})}\to (\mathcal{O}(\mathsf{T})\otimes\Lambda(\mathsf{Z}))^{\mathcal{O}(\mathsf{H})}$, which are dual to the open embeddings
$\mathbb{T}/\mathbb{H}\to \mathbb{V}/\mathbb{H}$.

Similarly, $\mathcal{O}(\mathsf{gr}(\mathbb{V}))\simeq (\mathcal{O}(\mathsf{U})\otimes\Lambda(\mathsf{Z}))\square_{\mathcal{O}(\mathsf{H})}\mathsf{gr}(\mathcal{O}(\mathbb{H}))$ and
$\mathcal{O}(\mathsf{gr}(\mathbb{V})/\mathsf{gr}(\mathbb{H}))\simeq (\mathcal{O}(\mathsf{U})\otimes\Lambda(\mathsf{Z}))^{\mathcal{O}(\mathsf{H})}$ (see \cite[Lemma 14.11]{maszub3}). It remains to note that all superalgebras
$(\mathcal{O}(\mathsf{U})\otimes\Lambda(\mathsf{Z}))^{\mathcal{O}(\mathsf{H})}$ and $(\mathcal{O}(\mathsf{T})\otimes\Lambda(\mathsf{Z}))^{\mathcal{O}(\mathsf{H})}$ are Grassman graded and use Remark \ref{graded}.

Finally, by \cite[Theorem 12.11]{maszub3} the second statement obviously holds, whenever $\mathbb{H}$ is normal. Again, arguing as in \cite[Lemma 14.2]{maszub3}, one can assume that $\mathbb{G}$ and $\mathbb{H}$ satisfy the above conditions (a) and (b). In particular, one can construct a covering of $\mathbb{G}/\mathbb{H}$ by open affine super-subschemes $\mathbb{V}/\mathbb{H}$ as above.

Thus $(\mathbb{V}/\mathbb{H})_{ev}\simeq\mathrm{SSp}(\overline{(\mathcal{O}(\mathsf{U})\otimes\Lambda(\mathsf{Z}))^{\mathcal{O}(\mathsf{H})}})\simeq \mathrm{SSp}(\mathcal{O}(\mathsf{U})^{\mathcal{O}(\mathsf{H})})$ and each open embedding $(\mathbb{T}/\mathbb{H})_{ev}\to (\mathbb{V}/\mathbb{H})_{ev}$ is dual to the algebra morphism $\mathcal{O}(\mathsf{U})^{\mathcal{O}(\mathsf{H})}\to \mathcal{O}(\mathsf{T})^{\mathcal{O}(\mathsf{H})}$. In other words, gluing together the purely even superschemes $(\mathbb{V}/\mathbb{H})_{ev}$ one obtains $\mathbb{G}_{ev}/\mathbb{H}_{ev}$. Proposition is proven.
\end{proof}

\section{Super-dimension of certain Noetherian superschemes}

Recall that if $X$ is an irreducible superscheme of finite type over a field $\Bbbk$, then its super-dimension is defined as the Krull super-dimension $\mathrm{Ksdim}(\mathcal{O}(U))$ of
the superalgebra of global sections (coordinate superalgebra) of any open (nonempty) affine super-subcheme $U$ in $X$ (see \cite{maszub}, Section 6.1). By analogy with this, we define the super-dimension of a Noetherian superscheme $X$ with $\dim(X_{res})<\infty$ as
\[
\mathrm{sdim}(X)=\sup_{U} \mathrm{Ksdim}(\mathcal{O}(U)), 
\] 
where $U$ runs over a covering of $X$ by open (nonempty) affine super-subschemes and the supremum is given with respect to the lexicographical order on couples of nonnegative integers (see \cite{maszub}, Section 5.3).
\begin{lm}\label{affine case}
If $X$ is affine, then for any covering of $X$ by open (nonempty) affine super-subschemes $U$ we have
\[  \sup_{U} \mathrm{Ksdim}(\mathcal{O}(U))=\mathrm{Ksdim}(\mathcal{O}(X)). \] 	
\end{lm}
\begin{proof}
Let $A$ denote $\mathcal{O}(X)$ and let $d$ denote $\dim(X_{res})$. Note that any affine open super-subscheme $U$ of $X$ has a form $\cup_{i\in I}\mathrm{SSpec}(A_{a_i})$, where
$I$ is a finite set, each element $a_i$ belongs to $A_0$ and $\sum_{i\in I} \mathcal{O}(U)_0 a_i|_U=\mathcal{O}(U)_0$ (cf. \cite[Lemma 3.5]{maszub2}). Thus all we need is to prove our lemma for any such finite covering of $X$.

So, let $X=\cup_{1\leq i\leq k}\mathrm{SSpec}(A_{a_i})$, where $\sum_{1\leq i\leq k}A_0a_i=A_0$. Since the underlying topological space of $X$ coincides with the underlying topological space of the scheme $X_{res}=\mathrm{Spec}(\overline{A})$, \cite[Exercise I.1.10]{hart} implies
\[\max_{1\leq i\leq k} \mathrm{Ksdim}_0(A_{a_i})=\mathrm{Ksdim}_0(A)=d.\]
Let $J$ denote the set $\{i\mid \mathrm{Ksdim}_0(A_{a_i})=d \}$.
It remains to show that
\[ 
\max_{i\in J} \mathrm{Ksdim}_1(A_{a_i})=\mathrm{Ksdim}_1(A). 
\]
If $\mathrm{Ksdim}_1(A)=t>0$, then let $y_1, \ldots, y_t\in A_1$ be a corresponding longest system of odd parameters. Choose a longest prime chain $\mathfrak{p}_0\subseteq \cdots\subseteq \mathfrak{p}_d$ in $A_0$ such that $\mathrm{Ann}_{A_0}(y_1\cdots y_t)\subseteq\mathfrak{p}_0$.
There is an index $i$ such that $a=a_i$ does not belong to $\mathfrak{p}_d$, hence
\[ 
(\mathfrak{p}_0)_a\subseteq \cdots\subseteq (\mathfrak{p}_d)_a 
\]
is the longest prime chain in $A_a$, and thus $i\in J$. Moreover, there is
\[\mathrm{Ann}_{(A_0)_a}(y_1\cdots y_t)=(\mathrm{Ann}_{A_0}(y_1\cdots y_t))_a\subseteq (\mathfrak{p}_0)_a  \]
and the latter inclusion implies $\mathrm{Ksdim}_1(A_{a})\geq \mathrm{Ksdim}_1(A)$.

Conversely, for any $i\in J$ there is a longest prime chain $\mathfrak{q}_0\subseteq \cdots\subseteq \mathfrak{q}_d$ in $A_0$
such that $(\mathfrak{q}_0)_{a_i}\subseteq\cdots \subseteq (\mathfrak{q}_d)_{a_i}$ is the longest prime chain in $A_{a_i}$.
If
\[
\textstyle 
y'_1=\frac{z_1}{a_i^{k_1}}, \quad \ldots, \quad y'_s=\frac{z_s}{a_i^{k_s}}
\] 
is the longest system of odd parameters in $A_{a_i}$ that subordinates
$(\mathfrak{q}_0)_{a_i}$, then again
\[ \mathrm{Ann}_{(A_0)_{a_i}}(y'_1\cdots y'_s)=(\mathrm{Ann}_{A_0}(z_1\cdots z_s))_{a_i}\subseteq (\mathfrak{q}_0)_{a_i} \]
implies  $\mathrm{Ann}_{A_0}(z_1\cdots z_s)\subseteq\mathfrak{q}_0$ and $\mathrm{Ksdim}_1(A_{a_i})\leq \mathrm{Ksdim}_1(A)$.

Finally, $\mathrm{Ksdim}_1(A)=0$ if and only if $\mathrm{Ann}_{A_0}(A_1)$ do not subordinate the first member of any longest prime chain in $A$. The same arguments as above show that $\mathrm{Ksdim}_1(A_{a_i})=0$ for any $i\in J$. Lemma is proven.
\end{proof}
\begin{pr}\label{well defined}
The definition of $\mathrm{sdim}(X)$ does not depend on the choice of covering of $X$ by affine open super-subschemes. 	
\end{pr}
\begin{proof}
Assume that $X$ has a covering by open affine super-subschemes $U_i, i\in I$, and by open afine super-subschemes $V_j, j\in J,$ as well.
For any couple of indices $(i, j)\in I\times J$ we choose a finite covering by open affine super-subschemes $Z_{i j t}, t\in T_{ij}$.
Then Lemma \ref{affine case} implies
\[
\begin{split}
\sup_{(i, j)\in I\times J, t\in T_{ij}} \mathrm{Ksdim}(\mathcal{O}(Z_{i j t}))&=\sup_{i\in I}(\sup_{j\in J, t\in T_{ij}}\mathrm{Ksdim}(\mathcal{O}(Z_{i j t})) )\\
&=\sup_{i\in I}\mathrm{Ksdim}(\mathcal{O}(U_i))  
\end{split}
\]
and symmetrically
\[
\begin{split}
\sup_{(i, j)\in I\times J, t\in T_{ij}} \mathrm{Ksdim}(\mathcal{O}(Z_{i j t}))&=\sup_{j\in J}(\sup_{i\in I, t\in T_{ij}}\mathrm{Ksdim}(\mathcal{O}(Z_{i j t})) )\\
&=\sup_{j\in J}\mathrm{Ksdim}(\mathcal{O}(V_j)).  
\end{split}
\]
\end{proof}
\begin{example}\label{algebraic group}
Let $G$ be an algebraic group superscheme. Then
\[\mathrm{sdim}(G)=(\dim(G_{res}), \dim(\mathrm{Lie}(G)_1)).\]	
As it has been already mentioned, the superscheme $G$ is isomorphic to $G_{res}\times \mathrm{SSpec}(\Lambda(\mathsf{V}^*))$, where
$\mathsf{V}=\mathrm{Lie}(G)_1$. Furthermore, if the open affine subschemes $U$ form a covering of $G_{res}$, then the open affine super-subschemes $W=U\times \mathrm{SSpec}(\Lambda(\mathsf{V}^*))$ form a covering of $G$. It remains to note that any basis of
$\mathsf{V}^*$ form a longest system of odd parameters of $\mathcal{O}(W)\simeq\mathcal{O}(U)\otimes\Lambda(\mathsf{V}^*)$, hence
\[\mathrm{Ksdim}(\mathcal{O}(W))=(\mathrm{Kdim}(\mathcal{O}(U)), \dim(\mathrm{Lie}(G)_1)).  \]
\end{example}
\begin{rem}\label{base extension}
If $X$ is an Noetherian superscheme with $\dim(X_{res})<\infty$, then for any field extension $\Bbbk \subseteq \Bbbk'$ there is
\[\mathrm{Ksdim}(X)=\mathrm{Ksdim}(X_{\Bbbk'}), \]	
where $X_{\Bbbk'}\simeq X\times_{\mathrm{SSpec}(\Bbbk)} \mathrm{SSpec}(\Bbbk')$. Indeed, without loss of a generality one can suppose that $X$ is affine, say $X\simeq\mathrm{SSpec}(A)$. Then $X_{\Bbbk'}\simeq \mathrm{SSpec}(A\otimes_{\Bbbk} \Bbbk')$ and it is obvious that
\[\mathrm{Ksdim}(A)=\mathrm{Ksdim}(A\otimes_{\Bbbk} \Bbbk').  \]
\end{rem}

\section{Super-dimension of a sheaf quotient}

Let $G$ be an algebraic group superscheme and $H$ be its closed group super-subscheme. Recall that $G$ corresponds to the Harisgh-Chandra pair $(G_{ev}, \mathsf{V})$,  where $\mathsf{V}=\mathrm{Lie}(G)_1$ is regarded as a $G_{ev}$-module with respect to the adjoint action, and the bilinear $G_{ev}$-invariant map $\mathsf{V}\times \mathsf{V}\to \mathrm{Lie}(G_{ev})=\mathrm{Lie}(G)_0$ is induced by the Lie superalgebra bracket.

\begin{tr}\label{super-dimension of a quotient}
We have $\mathrm{sdim}(G/H)=\mathrm{sdim}(G)-\mathrm{sdim}(H)$.	
\end{tr}
\begin{proof}
Proposition \ref{graded quotient} and \cite[Proposition 5.23]{milne} imply
\[
\begin{split}
\mathrm{sdim}_0(G/H)&=\dim((G/H)_{ev})=\dim(G_{ev}/H_{ev})\\
& =\dim(G_{ev})-\dim(H_{ev})=\mathrm{sdim}_0(G)-\mathrm{sdim}_0(H).  
\end{split}
\] 	
If $H$ is normal, then by \cite[Theorem 12.11]{maszub3} the statement of our proposition is obvious. Using the same reduction as in Proposition \ref{graded quotient}, based on arguments from \cite[Lemma 14.2]{maszub3}, one can assume that $\mathbb{G}$ and $\mathbb{H}$ satisfy the conditions (a) and (b) from Proposition \ref{graded quotient}.	It remains to show that for
any affine super-subscheme $\mathbb{V}$ constructed therein, we have $\mathrm{sdim}_1(\mathbb{V}/\mathbb{H})=\dim(\mathrm{Lie}(G)_1) -\dim(\mathrm{Lie}(H)_1)$. Let $B$ and $A$ denote
$\mathcal{O}(\mathbb{V})$ and $\mathcal{O}(\mathbb{V}/\mathbb{H})$ respectively.

Without loss of a generality one can assume that $\Bbbk$ is algebraically closed. In fact, for any field extension $\Bbbk\subseteq\Bbbk'$ the naive quotient $A\mapsto \mathbb{G}(A)/\mathbb{H}(A)$, $A\in\mathsf{SAlg}_{\Bbbk'}$, is apparently dense in $(\mathbb{G}/\mathbb{H})_{\Bbbk'}$ with respect the Grothendieck topology of fppf coverings, that is  $\mathbb{G}_{\Bbbk'}/\mathbb{H}_{\Bbbk'}\simeq (\mathbb{G}/\mathbb{H})_{\Bbbk'}$. Remark \ref{base extension} implies the required.

By \cite[Theorem 14.1]{maszub3} the quotient morphism $\mathbb{V}\to \mathbb{V}/\mathbb{H}$ is faithfully flat. Since $\mathbb{V}$ and
$\mathbb{V}/\mathbb{H}$ are affine, it is equivalent to the condition that $A\to B$ is faithfully flat. In particular, for any prime superideal $\mathfrak{P}$ of $A$ there is a prime superideal
$\mathfrak{Q}$ of $B$ such that $\mathfrak{Q}\cap A=\mathfrak{P}$. Moreover, the induced
local superalgebra morphism $A_{\mathfrak{P}}\to B_{\mathfrak{Q}}$ is flat (cf. \cite[Corollary 3.2]{maszub3}).

We choose $\mathfrak{P}$ so that its even component $\mathfrak{P}_0=\mathfrak{p}$ is the largest member of a longest prime chain of
$A_0$ (in particular, $\mathfrak{p}$ is a maximal ideal), and some longest system of odd parameter of $A$ subordinates the smallest prime ideal of this chain. Thus $\mathrm{Ksdim}_1(A)=
\mathrm{Ksdim}_1(A_{\mathfrak{P}})$. The superalgebra $B_{\mathfrak{Q}}\simeq\mathcal{O}(\mathsf{U})_{\mathfrak{q}}\otimes\Lambda(\mathsf{V}^*)$ is obviously oddly regular, thus by \cite[Proposition 3.6.1(iii)]{schmitt} the superalgebra $A_{\mathfrak{P}}$ is oddly regular as well. Next, by \cite[Proposition 3.6.3(ii)]{schmitt} we have
\[
\dim_{\Bbbk}(\Phi_{B_{\mathfrak{Q}}})=\dim_{\Bbbk}(\Phi_{A_{\mathfrak{P}}})+\dim_{\Bbbk}(B_{\mathfrak{Q}}/\mathfrak{P}B_{\mathfrak{Q}}), 
\]
or
\[
\mathrm{Ksdim}_1(\mathbb{V})=\mathrm{Ksdim}_1(\mathbb{V}/\mathbb{H}) +\dim_{\Bbbk}(B_{\mathfrak{Q}}/\mathfrak{P}B_{\mathfrak{Q}}). 
\]
Recall that the superideals $\mathfrak{Q}$ and $\mathfrak{P}$ can be interpreted as the \emph{(closed) points} $x\in\mathbb{V}(\Bbbk)$ and $y\in(\mathbb{V}/\mathbb{H})(\Bbbk)$.
As it has been observed in \cite[Theorem 7.5]{zubkol}, the superalgebra $B_{\mathfrak{Q}}/\mathfrak{P}B_{\mathfrak{Q}}$ is isomorphic to the stalk of the \emph{fiber} $\mathbb{V}_y\simeq \mathbb{V}\times_{\mathbb{V}/\mathbb{H}} \mathrm{SSp}(\Bbbk)$ at the point $x$, where, by virtue of Yoneda lemma, the morphism $\mathrm{SSp}(\Bbbk)\to \mathbb{V}/\mathbb{H}$ is uniquely defined by the point $y$. It is easy to see that $\mathbb{V}_y$ is naturally isomorphic to the left coset  $x\mathbb{H}$, whence to $\mathbb{H}$. Since any stalk of $\mathbb{H}$ is oddly regular with odd Krull dimension $\dim(\mathrm{Lie}(H)_1)$, our theorem follows.
\end{proof}

\section{Orbits}

Let $\mathbb{G}$ be an algebraic group superscheme that acts on a superscheme $\mathbb{X}$ of finite type (on the left). Take $x\in\mathbb{X}(\Bbbk)$ and consider the orbit $\mathbb{G}x\simeq\mathbb{G}/\mathbb{G}_x$. Below we prove that $\mathbb{G}x$ is a \emph{locally closed} super-subscheme of $\mathbb{X}$, provided $\Bbbk$ is algebraically closed.

A superscheme morphism $f:  X\to Y$ is called  a \emph{monomorphism}, provided ${\bf f}(A):  \mathbb{X}(A)\to \mathbb{Y}(A)$ is injective for any superalgebra $A$. For example, the superscheme morphism ${\bar a}_x:  G/G_x\to X$, induced by the orbit morphism $a_x$, is a monomorphism.
Also note that \cite[Lemma 5.5]{maszub2} implies  that $f^e$ is injective, whenever $f$ is a monomorphism.
\begin{lm}\label{G-subfunctor}
The orbit $\mathbb{G}x$ is a $\mathbb{G}$-stable (or, $\mathbb{G}$-saturated) subfunctor of $\mathbb{X}$.	
\end{lm}
\begin{proof}
An element $y\in \mathbb{X}(A)$ belongs to $(\mathbb{G}x)(A)$ if and only if there is a fppf covering $A'$ of $A$ such that $y'=\mathbb{X}(\iota_A)(y)$ belongs to $\mathbb{G}(A')x$. Thus for any $g\in\mathbb{G}(A)$ we have 	$\mathbb{X}(\iota_A)(gy)=g' y'\in\mathbb{G}(A')x$, hence $gy\in (\mathbb{G}x)(A)$.
\end{proof}
\begin{lm}\label{a partial case of the next theorem}
If $\mathrm{SSpec}(B) \to \mathrm{SSpec}(A)$ is a monomorphism of superschemes of finite type, induced by a superalgebra homomorphism $\phi:  A\to B$, then $B_1=B_0\phi(A_1)$.
\end{lm}
\begin{proof}
Let $C$ denote $B_0\oplus B_0\phi(A_1)$. The monomorphism $\mathrm{SSpec}(\phi)$ is a composition of $\mathrm{SSpec}(B)\to \mathrm{SSpec}(C)$ and
$\mathrm{SSpec}(C)\to\mathrm{SSpec}(A)$. Thus the first morphism is a monomorphism as well.

Let $\mathfrak{m}$ be a maximal ideal in $B_0$. Then $B/B\mathfrak{m}= L\oplus V$, where $L=B_0/\mathfrak{m}$ is a field extension of $\Bbbk$ and $V=B_1/B_1\mathfrak{m}$ is a finite-dimensional purely odd $L$-superspace such that $V^2=0$. If $W=(C_1 + B_1\mathfrak{m})/B_1\mathfrak{m}$ is a proper subspace of $V$, then choose a subspace $W'$ of $V$ of codimension $1$, such that $W\subseteq W'$. Observe that any (super)subspace of $V$ is a superideal of $B/B\mathfrak{m}$. Thus $R=L\oplus V/W'$ is a superalgebra and any linear map of $L$-superspaces $B/B\mathfrak{m}\to R$ is a superalgebra morphism.

Let $v$ be a basis vector of a complement of $W'$. Then
there are at least $|L|\geq 2$ different superalgebra morphisms from $B/B\mathfrak{m}$ to a superalgebra $R$, which are the same being restricted on the super-subalgebra $(C+B\mathfrak{m})/B\mathfrak{m}$. They are $\phi_a(W')=0, \phi_a(v)=av+W', a\in L$.
Furthermore, it also follows that the map $SSp \ B(R)\to SSp \ C(R)$ is not injective. This contradiction infers that $B_1=C_1+B_1\mathfrak{m}$, and hence, by Nakayama's lemma, $(B_1)_{\mathfrak{m}}=(C_1)_{\mathfrak{m}}$ for any maximal ideal $\mathfrak{m}$. Theorem 1 from \cite[Chapter II, \S 3]{bur},  implies $C_1=B_1$.
\end{proof}
\begin{pr}\label{local immersion}
Let $f:  X\to Y$ be a monomorphism of superschemes of finite type. Then there is an open dense subset $V\subseteq Y^e$ such that $U=(f^e)^{-1}(V)\neq\emptyset$ and $f|_U:  U\to V$ is an immersion.
\end{pr}
\begin{proof}
By \cite[Lemma 1.4 and Lemma 1.5]{maszub3} one needs to consider the case $X=\mathrm{SSpec}(B)$ and $Y=\mathrm{SSpec}(A)$ only. Then $f=\mathrm{SSpec}(\phi)$, where $\phi:  A\to B$ is a superalgebra morphism.
	
Since the induced morphism $f_{res}:  X_{res}=\mathrm{Spe}c(\overline{B})\to Y_{res}=\mathrm{Spec}(\overline{A})$ is also a monomorphism (in the category of schemes), \cite[I, \S 3, Corollary 4.7]{dg} implies that there is an open dense subset $V\subseteq Y^e=Y_{res}^e$ such that
$U=(f^e)^{-1}(V)\neq\emptyset$ and for any $x\in U$ the induced algebra morphism
\[ \mathcal{O}_{Y, y}/\mathcal{J}_y \to \mathcal{O}_{X, x}/\mathcal{J}_x,\]
where $\mathcal{J}_y=\mathcal{O}_{Y, y}(\mathcal{O}_{Y, y})_1$ and $\mathcal{J}_x=\mathcal{O}_{X, x}(\mathcal{O}_{X, x})_1$, $y=f^e(x)$, is surjective. Let $R, S$ and $\psi$ denote $\mathcal{O}_{Y, y}, \mathcal{O}_{X, x}$ and the induced local morphism $\mathcal{O}_{Y, y}\to \mathcal{O}_{X, x}$ respectively. By Lemma \ref{a partial case of the next theorem} there is $S_1=S_0 \psi(R_1)$. As it has been observed, the algebra morphism
\[\overline{R}\simeq R_0/R_1^2\to \overline{S}=S_0/S_1^2=S_0/S_0 \psi(R_1)^2\]
is surjective. The induction on $n$ infers that all space morphisms
\[R_1^{2n}/R_1^{2n+2}\to S_1^{2n}/S_1^{2n+2}=S_0\psi(R_1)^{2n}/S_0\psi(R_1)^{2n+2}\]
are also surjective for any $n\geq 1$. Since the $R_0$-ideal $R_1^2$ is nilpotent, $\psi|_{R_0}:  R_0\to S_0$ is surjective, hence $\psi$ is. Thus our proposition follows.
\end{proof}
The following lemma is a folklore.
\begin{lm}\label{underlying topological space}(see \cite[A.12]{milne})
Let $X$ be a superscheme of finite type, then $\mathbb{X}(\overline{\Bbbk})$ is dense in $X^e$. In particular, if $f:  X\to Y$ is a morphism of superschemes, then $f^e$ is surjective if and only if ${\bf f}(\overline{\Bbbk}):  \mathbb{X}(\overline{\Bbbk})\to \mathbb{Y}(\overline{\Bbbk})$ is.
\end{lm}
\begin{proof}
For any field extension $\Bbbk\subseteq L$ the set $\mathbb{X}(L)$ consists of all points $x\in X^e$ such that $\kappa(x)$ is a subfield of $L$. Without loss of a generality, one can assume that $X\simeq\mathrm{SSp}(A)$ and $x$ corresponds to a prime superideal $\mathfrak{P}$ of
$A$. Any open neighborhood $U$ of $x$ has a form $\{\mathfrak{Q}\mid \mathfrak{Q}\not\supseteq I  \}$ for a certain superideal $I$ of $A$. Since $A_0$ is a finitely generated $\Bbbk$-algebra and the ideal $I_0$ is not nilpotent, Hilbert Nullestellensatz implies that there is a maximal superideal $\mathfrak{M}$ such that $\mathfrak{M}\not\supseteq I$. If $y$ is the corresponding point in $X^e$, then $\kappa(y)$ is an algebraic extension of $\Bbbk$, whence $y\in\mathbb{X}(\overline{\Bbbk})\cap U$. The second statement is now obvious.
\end{proof}
\begin{tr}\label{orbit map is an immersion}
Let $\mathbb{X}$ be a superscheme of finite type and let an algebraic group superscheme $\mathbb{G}$ act on $\mathbb{X}$ on the left. If $\Bbbk$ is algebraically closed, then the orbit $\mathbb{G}x$ is a locally closed super-subscheme of $\mathbb{X}$ for any $x\in\mathbb{X}(\Bbbk)$. Moreover, if $\mathbb{G}$ is a smooth algebraic group superscheme (i.e. $G_{res}$ is a smooth scheme), then there are always closed orbits.	
\end{tr}
\begin{proof}
All we need is to prove that the morphism ${\bar a}_x:  G/G_x\to X$ is an immersion. By Proposition \ref{local immersion} there is a dense open super-subscheme $V\subseteq X$ such that the induced morphism
$U={\bar a}_x^{-1}(V)\to V$ is an immersion.

Let $\mathbb{U}$ and $\mathbb{V}$ denote the corresponding open super-subschemes of $\mathbb{G}/\mathbb{G}_x$ and $\mathbb{X}$, respectively.
Then $\mathbb{W}=\cup_{g\in\mathbb{G}(\Bbbk)} g\mathbb{U}$ is an open super-subscheme of $\mathbb{G}/\mathbb{G}_x$ and by \cite[Lemma 1.4]{maszub3}
the restriction of ${\bar {\bf a}}_x$ on $\mathbb{W}$ is an immersion. It remains to show that $\mathbb{W}=\mathbb{G}/\mathbb{G}_x$.
Let $\mathbb{W}'$ be the full inverse image of $\mathbb{W}$ in $\mathbb{G}$. Since $\mathbb{W}$ is obviously $\mathbb{G}(\Bbbk)$-stable,
we have $\mathbb{W}'(\Bbbk)=\mathbb{G}(\Bbbk)$, and Lemma \ref{underlying topological space} infers $\mathbb{W}'=\mathbb{G}$.
Thus $(\mathbb{G}/\mathbb{G}_x)_{(n)}\subseteq\mathbb{W}$ and since any open subfunctor is closed with respect to the Grothendieck topology of fppf coverings (just superize \cite[I.1.7(6)]{jan}), it follows that $\mathbb{W}=\mathbb{G}/\mathbb{G}_x$.

Since $\mathbb{X}(\Bbbk)=\mathbb{X}_{ev}(\Bbbk)$ and the super-subscheme $\mathbb{X}_{ev}$ is obviously $\mathbb{G}_{ev}$-stable, the restriction of ${\bf a}_x$ on $\mathbb{G}_{ev}$ coincides with the orbit morphism $\mathbb{G}_{ev}\to \mathbb{X}_{ev}$, $g\mapsto gx$, where  $g\in\mathbb{G}_{ev}(A)$, $A\in\mathsf{SAlg}_{\Bbbk}$. Furthermore, we have a commutative diagram
\[\begin{array}{ccc}
(\mathbb{G}/\mathbb{G}_x)_{ev} & \simeq & \mathbb{G}_{ev}/(\mathbb{G}_x)_{ev} \\
\downarrow & & \downarrow \\
\mathbb{X}_{ev} & = & \mathbb{X}_{ev}
\end{array},  \]
in which  $(\mathbb{G}_x)_{ev}=(\mathbb{G}_{ev})_x$, the left vertical arrow is $({\bar {\bf a}}_x)_{ev}$  and the right vertical arrow is unduced
by $({\bf a}_x)_{ev}$. It remains to note that an immersion $\mathbb{Y}\to\mathbb{Z}$ is a closed immersion if and only if the image of $Y^e$ is closed in $Z^e$, hence if and only if $Y_{ev}\to Z_{ev}$ is a closed immersion. The geometric counterpart of the above diagram is
\[\begin{array}{ccc}
(G/G_x)_{ev} & \simeq & G_{ev}/(G_x)_{ev} \\
\downarrow & & \downarrow \\
X_{ev} & = & X_{ev}
\end{array},  \]
where $(G_x)_{ev}=(G_{ev})_x$, the left vertical arrow is $({\bar a}_x)_{ev}$  and the right vertical arrow is unduced
by $(a_x)_{ev}$. Then \cite[II, \S 5, Poposition 3.3]{dg} concludes the proof.  	
\end{proof}
\begin{cor}\label{rem1}
In the conditions of the above theorem an orbit $\mathbb{G}x$ (respectively, $Gx$) is closed, whenever $\mathbb{G}_{ev}x$ (respectively, $G_{ev}x$) is, or equivalently, whenever $G_{res}x$ is. The latter takes place if $\dim(G_{res}x)=\mathrm{sdim}_0(Gx)$ is minimal. 	
\end{cor}
\begin{cor}\label{rem2}
Theorem \ref{super-dimension of a quotient} implies that $\mathrm{sdim}(Gx)=\mathrm{sdim}(G)-\mathrm{sdim}(G_x)$.	
\end{cor}

\section{An Example} 

We still assume that $\Bbbk$ is algebraically closed.

Let $X$ be an affine superscheme and let $G$ be an affine algebraic group superscheme. Then (left) actions of $G$ on  $X$ are in one-to-one correspondence with superalgebra morphisms $\mathcal{O}(X)\to \mathcal{O}(G)\otimes\mathcal{O}(X)$ so that $\mathcal{O}(X)$ becomes a $\mathcal{O}(G)$-supercomodule.

Corollary \ref{rem1} infers that if $G_{ev}=1$, then any orbit is closed. Moreover, the underlying topological space of the orbit of a point $x$ is just $\{x\}$. Below we illustrate this fact by the following elementary example.

The condition $G_{ev}=1$ implies that $G$ is isomorphic to the direct product of several copies of the odd unipotent group superscheme $G_a^-$ of super-dimension $0|1$. For the sake of simplicity we assume that $G=G_a^-$. The Hopf superalgebra $\mathcal{O}(G)$
has a basis $1, z$, where $z$ is an odd primitive element. The corresponding supercomodule map $\tau:  \mathcal{O}(X)\to \mathcal{O}(G)\otimes\mathcal{O}(X)$ is defined as
\[
f\mapsto 1\otimes f +z\otimes \phi(f), \qquad f\in \mathcal{O}(X), 
\]
where $\phi:  \mathcal{O}(X)\to \mathcal{O}(X)$ is an odd (left) superderivation with $\phi^2=0$.

Choose $x\in\mathbb{X}(\Bbbk)$. For any superalgebra $A$ and arbitrary element $g\in \mathbb{G}(A)$ the element $gx$ is a
superalgebra morphism $\mathcal{O}(X)\to A$ such that
\[
(\star) \ (gx)(f)=x(f)+g(z)x(\phi(f)),\qquad  f\in \mathcal{O}(X).  
\]
Recall that $\ker x=\mathfrak{M}=\mathfrak{m}\oplus\mathcal{O}(X)_1$ is a maximal superideal in $\mathcal{O}(X)$  with $\mathcal{O}(X)_0/\mathfrak{m}=\Bbbk$. Set $I=\mathfrak{m}\oplus\phi^{-1}(\mathfrak{m})$.
\begin{lm}\label{all orbits are closed}	
The subspace $I$ is a superideal in $\mathcal{O}(X)$. Furthermore, the orbit $\mathbb{G}x$ is closed and isomorphic to $\mathrm{SSp}(\mathcal{O}(X)/I)$.	
\end{lm}
\begin{proof}
All we need is to prove that $\mathcal{O}(X)_1\mathfrak{m}\subseteq\phi^{-1}(\mathfrak{m})$ and $\mathcal{O}(X)_0\phi^{-1}(\mathfrak{m})\subseteq\phi^{-1}(\mathfrak{m})$. We have
\[\phi(\mathcal{O}(X)_1\mathfrak{m})\subseteq \mathcal{O}(X)_0\mathfrak{m}+\mathcal{O}(X)_1^2\subseteq \mathfrak{m}.  \]
The proof of the second statement is similar. The equation $(\star)$ implies $(gx)(I)=0$. In other words, $\mathbb{G}x$ is a subfunctor
of the closed super-subscheme $\mathbb{Y}$ of $\mathbb{X}$, which is isomorphic to $\mathrm{SSp}(\mathcal{O}(X)/I)$. Since $\tau(I)\subseteq \mathcal{O}(G)\otimes I$, $\mathbb{Y}$ is $\mathbb{G}$-stable and $x\in\mathbb{Y}(\Bbbk)$.

If $\phi^{-1}(\mathfrak{m})=\mathcal{O}(X)_1$, then $\mathbb{G}_x=\mathbb{G}$ and $\mathbb{Y}=\mathbb{G}x=\{x\}$. Otherwise, $\mathcal{O}(X)_1/\phi^{-1}(\mathfrak{m})$ is one dimensional. Let $v$ be a basis vector of $\mathcal{O}(X)_1/\phi^{-1}(\mathfrak{m})$ such that $\phi$ takes $v$ to $1\in\Bbbk=\mathcal{O}(X)_0/\mathfrak{m}$. Then for any $\psi\in \mathbb{Y}(A)$ there is $\psi=gx$, where
$g(z)=\psi(v)$. Lemma is proven.
\end{proof}

\end{document}